\documentclass[a4paper,12pt]{article}

\usepackage{amsmath,amsthm,amssymb}
\usepackage{enumerate}
\usepackage{graphicx}
\usepackage{titlesec}
\usepackage{empheq}
\usepackage{color}
\usepackage{xfrac,bigints}
\usepackage{comment}
\usepackage{subfigure}
\usepackage{caption}
\usepackage[normalem]{ulem}

\newtheorem{Th}{Theorem}[section]
\newtheorem{lemma}[Th]{Lemma}
\newtheorem{remark}[Th]{Remark}

\newtheorem{a priori condition}[Th]{A priori condition}

\newtheorem{assumption}[Th]{Assumption}

\titleformat*{\section}{\normalsize\bfseries}
\titleformat*{\subsection}{\normalsize\bfseries}
\renewcommand{\thesection}{{{\normalsize\arabic{section}}}}

\renewcommand{\theequation}{{{\thesection.\arabic{equation}}}}

\newcommand{\pf}{\noindent{\bf Proof}}

\def\R{{\mathbb R}}

\def\DS{\displaystyle}

\def\dvg{\mbox{\rm div}}

\definecolor{purple}{rgb}{.75, 0, .25}

\makeatletter
\long\def\@makefntext#1{\parindent 1em\noindent
\@hangfrom{\hbox to 1.8em{\hss$^{\@thefnmark}$}}#1}
\makeatother

\begin{document}

\title{Explicit form of relaxation tensor for isotropic extended Burgers model and its spectral inversion}

\author{
Youjun Deng\thanks{School of Mathematics and Statistics, Central South University,Changsha 410083, P.R. China (youjundeng@csu.edu.cn).}
\and
Ching-Lung Lin\thanks{Department of Mathematics, National Cheng-Kung University, Tainan 701, Taiwan (Email:
cllin2@mail.ncku.edu.tw).}
\and Gen Nakamura\thanks{Department of
Mathematics, Hokkaido University, Sapporo 060-0808, Japan and Research Center of Mathematics for Social Creativity, Research Institute for Electronic Science, Hokkaido University, Sapporo 060-0812, Japan (Email: gnaka@math.sci.hokudai.ac.jp).}}

\date{}
\maketitle

\begin{abstract}
Concerning the anelastic nature of Earth, the quasi-static extended Burgers model (abbreviated by q-EBM), an integro-differential system, is used to study the free oscillation of Earth (abbreviated by FOE). In this paper, we first provide a general method to obtain an explicit form of the relaxation tensor for inhomogeneous isotropic q-EBM. Then, we apply it to compute the eigenvalues of the free oscillation of Earth, assuming that Earth is a unit ball modeled as a homogeneous and isotropic q-EBM. So far, an analytical and systematic way to compute the eigenvalues of the FOE has been missing when modeling Earth as a q-EBM. In particular, we compute some clusters of eigenvalues (abbreviated by C-ev's). To be more precise, integrating by parts with respect to time of the q-EBM under the assumption that the initial strain is zero, the q-EBM becomes the sum of two terms. 
The first term, called the instantaneous term, doesn't have any integration with respect to time, but the second term, called the memory term, has such an integration.  Then, consider the eigenvalues of the instantaneous part of the q-EBM. The eigenfunctions of C-ev's share the same eigenfunctions of the instantaneous part. However, the C-ev's may be shifted from the eigenvalues of the instantaneous part. Further, we analyze the structure of C-ev's and provide an inversion formula identifying the q-EBM from the C-ev's. 

\bigskip


\end{abstract}

\renewcommand{\theequation}{\thesection.\arabic{equation}}

\section{Introduction}\label{introduction}
\label{sec1}
\setcounter{equation}{0}
This paper aims to study clusters of eigenvalues referred to as the C-ev's of the free oscillation of the Earth (abbreviated by FOE), which we model as a homogeneous isotropic extended Burgers model (abbreviated by EBM). In \cite{YP}, the importance of the EBM emerged in geophysics. The authors of this paper argue that this model can capture the full range of observed mechanical behavior of upper mantle material and reconcile it with the behavior of the Earth's mantle as a whole. For other applications of EBM in geophysics, see \cite{Ivins, LauFaul_2019, Harry} and the references therein. As the first attempt, we provide an analytic method to generate the C-ev's and analysis examining the structure of C-ev's.  

To begin with, we first give a bird's eye view of the C-ev's of FOE before getting into details. The equation of motion for the EBM yields the Boltzmann-type viscoelastic system (abbreviated by BVS). This is an integro-differential system given as the acceleration term 
equal to the divergence of stress over each time interval $[0,t]$ for $t\ge0$. Here, the stress is given as a convolution with a relaxation tensor, which is convoluted with the time derivative of strain. Assuming the strain is zero at the initial time, $t=0$, and integrating the BVS by parts with respect to time, it becomes that the acceleration term is equal to the sum of the elastic term and the memory term. The former and latter terms are referred to as the instantaneous term and memory term, respectively. For simplicity of notation, we assume the density of the EBM is equal to 1 throughout this paper.

At this stage, we replace the two times derivative with respect to time in the BVS by $z^2$-multiplication with $z\in{\mathbb C}$ combined with the traction-free boundary condition on the boundary $\partial\Omega$ of the Earth $\Omega$, assuming it as an open ball in ${\mathbb R}^3$ with radius $R>0$ centered at the origin for simplicity. Then, we have an eigenvalue problem with a parameter $t\ge0$. In this problem, we refer to the integro-differential operator that does not contain the parameter $z$ as a quasi-static EBM (abbreviated by q-EBM). To describe the C-ev's, we first consider the eigenvalues of the boundary value problem for the mentioned instantaneous term in $\Omega$ with the traction-free boundary condition on $\partial\Omega$. Then, for the eigenvectors associated with these eigenvalues, the C-ev's look for the eigenvalues for the mentioned eigenvalue problem for the q-EBM. For further arguments, it is important to note that the elastic tensors of springs in the EBM are commutative because they are isotropic. Based on this, we will show later that these eigenvectors are also the eigenvectors for the eigenvalue problem for the q-EBM with eigenvalues which are the roots of a polynomial of $z$. This is the advantage of capturing the C-ev's because they are roots of a polynomial and we can analyze their structure. 

For simplicity of notation, we assume that the density of $\Omega$ equals $1$ through out the rest of this paper. Before closing this section, we describe the EBM and BVS in details. Figure given after References schematically describes the EBM. Although the EBM is a one space dimensional model, it can be formally generalized to higher space dimensions. The mathematical meaning of the generalization can be given using the spectral decompositions of the elasticity tensors of the springs in the models (see \cite{HKLNT_IPMS, KLLNY} and the references there in). Also, for the isotropic case, there is a mechanical meaning of generalization given in \cite{Flugge} for $n=1$, which is for the Burgers model and yet for the EBM.

To proceed further, we denote a displacement vector for a small deformation of $\Omega$ by $u=u(x,t)\in\R^3$ for a position a.e. $x\in \Omega$ at a time $t\in [0,T)$, where $T\in (0,\infty ]$. We also define the infinitesimal strain tensor by $e[u]:=\frac{1}{2} (\nabla u +(\nabla u)^\mathfrak{t})$
$:~\Omega\times [0,T)\to \R^{3\times 3}_\text{sym}$ with the set ${\mathbb R}^{3\times3}_{\text{sym}}$ of all symmetric matrices and the notation $\mathfrak{t}$ for transposing matrices. Also, the following symbols will be used below.
For matrices $A=(a_{pq}),~B=(b_{pq})\in \R^{3\times 3}$, their inner product is defined by
$A:B:=a_{pq}b_{pq}$, where Einstein's summation symbol is used. For a second-order tensor $\alpha=(\alpha_{pq})\in \R^{3\times 3}$ and
a fourth-order tensor
$C=(c_{pqrs})\in \R^{3\times 3\times 3\times 3}$,
their product is denoted by
$C\alpha:=(c_{pqrs}\alpha_{rs})\in\R^{3\times 3}$.
The product of two fourth order tensors $C=(c_{pqrs})$ and $D=(d_{pqrs})$ is defined by $CD=(c_{pqkl}d_{klrs})$ as well.

Now for the EBM model in Figure, we denote its pair of strain and stress as $(e_i,\sigma_i)\in \R^{3\times 3}_{\text{sym}}\times \R^{3\times 3}_{\text{sym}}$ in accordance with the labeling number $i=0,1,\cdots,n$. Also, we denote by $(e,\sigma)\in \R^{3\times 3}_{\text{sym}}\times \R^{3\times 3}_{\text{sym}}$, the pair of strain and stress of the EBM. Then, using Hooke's law for springs and the fact that the stress of each dashpot is proportional to the instantaneous speed of its strain, the equation of motion for the EBM satisfying the traction-free boundary condition over $\partial\Omega$ is given as follows (see \cite{HKLNT} for the details). 
\begin{equation}\label{s1.1}
\left\{
\begin{array}{ll}
\ddot{u}=\dvg\, \sigma &\mbox{\rm in}~\Omega\times {\mathbb R},\\
\eta_0\dot{\phi}_0=\sigma &    \mbox{\rm in}~\Omega\times {\mathbb R},\\
\eta_i\dot{\phi}_i=\sigma -C_i\phi_i, \quad i=1,\cdots,n &    \mbox{\rm in}~\Omega\times {\mathbb R},\\
\sigma\nu=0 &\mbox{\rm on}~\partial\Omega\times {\mathbb R},\\
\end{array}
\right.
\end{equation}
with the divergence operator $\text{div}$, the unit outernormal of $\partial\Omega$, and the total stress tensor $\sigma$ defined by
\begin{align}\label{s1.2}
\sigma :=C_0\psi,\quad
  \psi:=
  e[u]-\sum_{i=0}^n\phi_i=e_0^s,
\end{align}
where for each $i\,(0\le i\le n)$ corresponding to the label $i$ of Figure. Also, $\eta_i>0$ and $C_i=((C_i)_{pqrs}):=({\lambda}_i\delta_{pq}\delta_{rs}+{\mu}_i(\delta_{pr}
\delta_{qs}+\delta_{ps}\delta_{qr}))$ with Kronecker's delta $\delta_{ij}$ are the inhomogeneous viscosity and elasticity tensor of the dashpot and spring, respectively. Whenever we use the terminology inhomogeneous in this paper, the object depending on the space variable $x\in\Omega$ is bounded measurable in a bounded Lipschitz domain $\Omega\subset{\mathbb{R}}^3$, and for notational simplicity, we suppress $x$. For the FOE, $\Omega$ is the mentioned ball, and the physical objects there are all homogeneous. Also, for each $\phi_i$ with $i=0,\cdots,n$ is the strain of the $i$-th dashpot, and we denote $\phi:=(\phi_1,\cdots,\phi_n)$, 
and the time derivative by $\dot{}$ for notational convenience. Here, we assume that the Lam\'e moduli $\lambda_i,\,\mu_i$ satisfy the strong convexity condition:
\begin{equation}\label{strong convex}
\mu_i\ge\delta,\,\,3\lambda_i+2\mu_i\ge\delta,\,\,i=0,\cdots,n
\end{equation}
for some constant $\delta>0$.
Next, by eliminating $(\phi_1,\cdots,\phi_n)$ using the second, third equations of \eqref{s1.1} and assuming $(\phi_1,\cdots,\phi_n)=0$ at $t=0$, we uniquely derive a relaxation tensor $G(t)$ satisfying the causality, $G(t)=0$ for $t<0$, also it is completely monotone fourth order tensor acting on $L^\infty(\Omega;{\mathbb R}^{3\times3}_{\text{sym}})$, the set of ${\mathbb R}^{3\times3}_{\text{sym}}$ valued bounded measurable functions in $\Omega$, over the time interval $[0,\infty)$ satisfying $G(0)=C_0$, and the stress-strain relation can be given as
\begin{equation}\label{stress_strain rel}
\sigma(t,x)=\int_0^t G(t-s)e[\dot{u}](s,x)\,ds
\end{equation}
(see \cite{HKLNT, Zemanian} for the details by viewing the EBM as a one-port mechanical system). Here, note that from the causality of $G(t)$ and \eqref{stress_strain rel}, we could have assumed $\psi=0$ at $t=0$. Henceforth, we will always take this into account, so that we assume
\begin{equation}\label{zero initial strains}
\phi:=(\phi_1,\cdots,\phi_n)=0,\,\,\psi=0\,\,\,\text{at}\,\,\,t=0.
\end{equation}

Then, combining this with the first equation of \eqref{s1.1}, we have
the BVS as follows:
\begin{equation}\label{Boltzmann_1}
\ddot{u}(t,x)=\dvg\int_0^t G(t-s)e[\dot{u}](s,x)\,ds.   \end{equation}
Further, assuming $e[u](t,x)=0$ at $t=0$, we have
\begin{equation}\label{Boltzmann_2}
\ddot{u}=\dvg\{C_0 e[u](t,x)\}+\int_0^t \dot{G}(t-s)e[u](s,x)\,ds.
\end{equation}
with the instantaneous term $\dvg\{C_0 e[u](t,x)\}$ and the q-EBM operator $\dvg\{C_0 e[u](t,x)\}+\int_0^t \dot{G}(t-s)e[u](s,x)\,ds$ acting to $u(t,x)$.

\medskip
The rest of this paper is organized as follows.
In the next section, we focus on the method giving an explicit form of the relaxation kernel of the inhomogeneous isotropic EBM defined in $\Omega\subset{\mathbb R}^3$. Hence, in these two sections, the elasticity tensors and viscosities can depend on $x\in\Omega$, and they are bounded measurable in $\Omega$. Section 3 is devoted to finding some special radial eigenfunctions for the instantaneous term which are also eigenfunctions of the q-EBM. In Section 4, we provide a method to generate C-ev's for the FOE. Then in Section 5, we analytically analyze the structure of C-ev's for the FOE and give an inversion formula identifying the q-EBM by knowing the C-ev's. The final section, Section 6, is for the conclusions and discussions.

\section{Derivation of the relaxation tensor for inhomogeneous isotropic q-EBM }\label{formula I-EBM}
\setcounter{equation}{0}
In this section, we derive an explicit form of the relaxation tensor for inhomogeneous isotropic EBM. Note that we already gave a general method to derive the relaxation tensor for inhomogeneous anisotropic EBM in \cite{HKLNT}. However, the form of the relaxation tensor was not explicit enough to study the FOE. We would like to emphasize that this paper gave an explicit form of the relaxation tensor for three-dimensional inhomogeneous isotropic EBM for the first time. For the derivation given below, we use the volumetric/deviatoric decompositions and easily computable symmetric tensor block matrices \eqref{s2.11}/\eqref{s2.13} associated with the ordinary differential systems \eqref{s2.8}/\eqref{s2.9} for the volumetric/deviatoric $\psi$, $\phi_i$'s. We note that without using the decompositions, the derivation quickly hits a snag.

To begin with, we first introduce volumetric projection $I_m$ and deviatoric projection $J_m$ acting on $L^\infty(\Omega;{\mathbb R}^{3\times3}_{\text{sym}})$ as multiplication by 4th order tensors which synchronize well with the EBM under consideration. They are commutative projection operators on $L^\infty(\Omega;{\mathbb R}^{3\times3}_{\text{sym}})$ defined as multiplication by the following fourth-order tensors
\begin{equation}\label{I-J}
I_m:=(3^{-1}\delta_{ij}\delta_{kl}),\,\,J_m:=\tilde I-I_m,    
\end{equation}
where $\tilde I:=I_3\otimes I_3$ is a fourth-order tensor with the $3\times3$ identity matrix $I_3$. Here, the multiplications between $\tilde I$ and the second-order tensors, fourth-order tensors follow those for tensors defined in Section \ref{introduction}. Since $I_m\,J_m=0$, the ranges of $I_m$ and $J_m$ are orthogonal.
Then, for $0\leq j\leq n$, we have
\begin{equation}\label{s2.2}
\left\{
\begin{array}{l}
\phi^V_j:=I_m\, \phi_j=3^{-1}{\rm tr} (\phi_j)\,I_3,\\
\phi^D_j:=J_m\,\phi_j=\phi_j-\phi^V_j
\end{array}
\right.
\end{equation}
and
\begin{equation}\label{s2.3}
\left\{
\begin{array}{l}
\psi^V:=I_m\,\psi==3^{-1}{\rm tr} (\psi)\,I_3,\\
\psi^D:=J_m\,\psi=\psi-\psi^V,
\end{array}
\right.
\end{equation}
where $I_3$ is the $3\times3$ identity matrix, and $\text{tr}(\phi_j)$ denotes the trace of $\phi_j$.

Now, combining \eqref{s1.1} and \eqref{s1.2}, we have
\begin{equation}\label{s2.1}
\left\{
\begin{aligned}
\partial_t\psi&=e[\dot{u}]-(\sum_{i=0}^n\eta_i^{-1})C_0\psi+\sum_{i=1}^n\eta_i^{-1}C_i\phi_i\\
\partial_t\phi_i&=\eta_i^{-1}C_0\psi-\eta_i^{-1}C_i\phi_i, \quad 1\leq i\leq n.
\end{aligned}
\right.
\end{equation}
Rewrite \eqref{s2.1} as
\begin{equation}\label{s2.4}
\left\{
\begin{aligned}
\partial_t\psi^V+\partial_t\psi^D=&e[\dot{u}]-bC_0\psi+\sum_{i=1}^n\eta_i^{-1}C_i\phi_i\\
=&e[\dot{u}]-b(3\lambda_0+2\mu_0)\psi^V-2b\mu_0\psi^D\\
&+\sum_{i=1}^n\eta_i^{-1}(3\lambda_i+2\mu_i)\phi_i^V+\sum_{i=1}^n2\mu_i\eta_i^{-1}\phi_i^D,\\
\partial_t\phi_i^V+\partial_t\phi_i^D=&\eta_i^{-1}(3\lambda_0+2\mu_0)\psi^V+2\eta_i^{-1}\mu_0\psi^D\\
&-\eta_i^{-1}(3\lambda_i+2\mu_i)\phi_i^V-2\mu_i\eta_i^{-1}\phi_i^D, \quad 1\leq i\leq n.
\end{aligned}
\right.
\end{equation}
Note here that we have
\begin{equation}\label{I_J e}
I_m\,e[\dot u]=3^{-1}(\nabla\cdot \dot u)I_3,\,\, J_m\,e[\dot u]=e[\dot u]-3^{-1}(\nabla\cdot \dot u)I_3.
\end{equation}

\noindent
Then, from \eqref{s1.1}, \eqref{s2.4},  \eqref{I_J e} and the orthogonality of the ranges of $I_m$ and $J_m$, we have \begin{equation}\label{s2.6}
\left\{
\begin{aligned}
&\partial_t\psi^V=3^{-1}(\nabla\cdot\dot{u}) I_3-b(3\lambda_0+2\mu_0)\psi^V+\sum_{i=1}^n\eta_i^{-1}(3\lambda_i+2\mu_i)\phi_i^V,\\
&\partial_t\phi_i^V=\eta_i^{-1}(3\lambda_0+2\mu_0)\psi^V-\eta_i^{-1}(3\lambda_i+2\mu_i)\phi_i^V,\\
&\psi^V(0)=0=\phi_i^V(0),\quad 1\leq i\leq n
\end{aligned}
\right.
\end{equation}
and
\begin{equation}\label{s2.7}
\left\{
\begin{aligned}
&\partial_t\psi^D=e[\dot{u}]-3^{-1}(\nabla\cdot\dot{u}) I_3-2b\mu_0\psi^D+\sum_{i=1}^n2\mu_i\eta_i^{-1}\phi_i^D,\\
&\partial_t\phi_i^D=2\mu_0\eta_i^{-1}\psi^D-2\mu_i\eta_i^{-1}\phi_i^D,\\
&\psi^D(0)=0=\phi_i^D(0),\quad 1\leq i\leq n.
\end{aligned}
\right.
\end{equation}
Solving \eqref{s2.6} for $\psi^V$, $\phi_i^V$'s, we have 
\begin{equation}\label{s2.8}
\begin{pmatrix}
\psi^V \\
\phi_1^V\\
\cdots  \\
\phi_n^V
\end{pmatrix}
(t)=\int_0^t e^{(t-s)L_0^{\tilde I}}
\begin{pmatrix}
3^{-1}(\nabla\cdot\dot{u})I_3 \\
0\\
\cdots\\
0
\end{pmatrix}
(s)\, ds,
\end{equation}
where
\begin{equation}\label{s2.9}
L_0^{\tilde I}:=
\begin{pmatrix}
                       -b(3\lambda_0+2\mu_0)\tilde I
                       & \eta_1^{-1}(3\lambda_1+2\mu_1)\tilde I & \cdots               &    \eta_n^{-1}(3\lambda_n+2\mu_n) \tilde I   \\

                             \eta_1^{-1}(3\lambda_0+2\mu_0)\tilde I                              & -\eta_1^{-1}(3\lambda_1+2\mu_1)\tilde I & \cdots                  &            0         \\
            \cdots                                   & \cdots                   & \cdots     & \cdots                            \\
                               \eta_{n-1}^{-1}(3\lambda_0+2\mu_0) \tilde I                               & 0              & \cdots          &            0         \\
                              \eta_n^{-1}(3\lambda_0+2\mu_0)\tilde I                              & 0              & \cdots                    &   -\eta_n^{-1}(3\lambda_n+2\mu_n)\tilde I
\end{pmatrix}.
\end{equation}
\noindent
For notational simplicity, we will suppress $\tilde I$ inside the matrix of \eqref{s2.9}, and denote the suppressed one as $L_0$. We remark that the relation between the two is not merely notational, but also a mathematical relation. In fact, consider an $\left(3(n+1)\right)^2$-dimensional Hermitian space $E$ defined by
$$
E:=\{\text{tensor block matrix}\,(\alpha_{ij}\tilde I)_{0\le i,j\le n}:\alpha_{ij}\in{\mathbb R},\,0\le i,\,j\le n\}
$$
with the canonical orthogonal basis consisting of tensor block matrices $\mathcal{E}_{ij}$ for $0\le i,j\le n$ such that each $\mathcal{E}_{ij}$'s $(k,l)$-tensor block component is $\beta_{kl}\tilde I$ with 
$$
\beta_{kl}=\left\{
\begin{array}{ll}
\delta_{ij}\,\,&\text{if}\,\,(k,l)
=(i,j),\\
0\,\,&\text{if otherwise}.
\end{array}
\right.
$$
We note here that the inner product for tensor block matrices is naturally defined as follows. It is the summation of tensor block componentwise inner products, and each of these inner products is the usual componentwise inner product for tensors. 
Then, $L_0^{\tilde I}$ and $L_0$ are algebraically and topologically equivalent.
Similarly, from \eqref{s2.7}, we have
\begin{equation}\label{s2.10}
\begin{pmatrix}
\psi^D \\
\phi^D
\end{pmatrix}
(t)=\int_0^t e^{(t-s)L_1^{\tilde I}}
\begin{pmatrix}
e[\dot{u}]-3^{-1}(\nabla\cdot\dot{u}) I_3 \\
0
\end{pmatrix}
(s)\, ds,
\end{equation}
where
\begin{equation}\label{s2.11}
L_1^{\tilde I}:=
\begin{pmatrix}
                       -2b\mu_0\tilde I
                       & \eta_1^{-1}2\mu_1 \tilde I & \cdots     & \eta_{n-1}^{-1}2\mu_{n-1}  \tilde I          &    \eta_n^{-1}2\mu_n \tilde I   \\

                             \eta_1^{-1}2\mu_0  \tilde I                            & -\eta_1^{-1}2\mu_1 \tilde I& \cdots     &    0               &            0         \\
            \cdots                                   & \cdots                   & \cdots     & \cdots             & \cdots               \\
                               \eta_{n-1}^{-1}2\mu_0    \tilde I                           & 0              & \cdots     & -\eta_{n-1}^{-1}2\mu_{n-1}   \tilde I  &            0         \\
                              \eta_n^{-1}2\mu_0      \tilde I                        & 0              & \cdots     &    0               &    -\eta_n^{-1}2\mu_n \tilde I
\end{pmatrix}.
\end{equation}

Next, we aim to compute $g_{00}(t)$ and analyze its property of the tensor block matrix $e^{tL_1^{I_3}}$ given as
\begin{equation}\label{s2.12}
e^{tL_1^{\tilde I}}=
\begin{pmatrix}
g_{00}(t)\tilde I&g_{01}(t)\tilde I&\cdots&g_{0n}(t)\tilde I\\
\ldots& & &\ldots\\
g_{n0}(t)\tilde I&g_{n1}(t)\tilde I&\cdots&g_{nn}(t)\tilde I
\end{pmatrix}.
\end{equation}
For that consider 
\begin{equation}\label{s2.13}
L_1:=
\begin{pmatrix}
                       -2b\mu_0
                       & \eta_1^{-1}2\mu_1 & \cdots     & \eta_{n-1}^{-1}2\mu_{n-1}            &    \eta_n^{-1}2\mu_n    \\

                             \eta_1^{-1}2\mu_0                              & -\eta_1^{-1}2\mu_1 & \cdots     &    0               &            0         \\
            \cdots                                   & \cdots                   & \cdots     & \cdots             & \cdots               \\
                               \eta_{n-1}^{-1}2\mu_0                               & 0              & \cdots     & -\eta_{n-1}^{-1}2\mu_{n-1}     &            0         \\
                              \eta_n^{-1}2\mu_0                              & 0              & \cdots     &    0               &    -\eta_n^{-1}2\mu_n
\end{pmatrix},
\end{equation}
and compute the $(1,1)$ component of $e^{t L_1}$. Then, it is enough to compute the $(1,1)$ component of the tensor block matrix $e^{t L_1^S}$ and multiply it to $\tilde I$, where $L_1^S$ is given as

\begin{equation}\label{s2.14}
\begin{array}{ll}
L^S_1:=\\
\\
\begin{pmatrix}
                       -b2\mu_0
                       &\eta_1^{-1}\left((2\mu_0)(2\mu_1)\right)^{1/2} & \cdots     & \eta_{n-1}^{-1}\left((2\mu_0)(2\mu_{n-1})\right)^{1/2}            &    \eta_n^{-1}\left((2\mu_0)(2\mu_{n})\right)^{1/2}    \\

                             \eta_1^{-1}\left((2\mu_0)(2\mu_1)\right)^{1/2}                             & -\eta_1^{-1}2\mu_1 & \cdots     &    0               &            0         \\
            \cdots                                   & \cdots                   \cdots     & \cdots             & \cdots               \\
                               \eta_{n-1}^{-1}\left((2\mu_0)(2\mu_{n-1})\right)^{1/2}                                & 0              & \cdots     & -\eta_{n-1}^{-1}2\mu_{n-1}     &            0         \\
                             \eta_n^{-1}\left((2\mu_0)(2\mu_{n})\right)^{1/2}                               & 0              & \cdots     &    0               &   - \eta_n^{-1}2\mu_n
\end{pmatrix}.
\end{array}
\end{equation}

\noindent
\begin{lemma}\label{s_lem2.1}
$L_1^S$ is negative definite.
\end{lemma}
\begin{proof}
Let $Y=(y_0,y_1,\cdots,y_n)^t$. Then, direct computations of the inner product $(L_1^SY,Y)$ give that
\begin{equation}\label{s2.15}
\begin{aligned}
(L_1^SY,Y)=&-\eta_02\mu_0y^2_0-\sum_{i=1}^n\eta_i^{-1}(\sqrt{2\mu_0}y_0-\sqrt{2\mu_i}y_i)^2\\
\leq &-c|Y|^2
\end{aligned}
\end{equation}
for some constant $c>0$.

\end{proof}

\begin{lemma}\label{s_lem2.2}
$g_{00}(t)$ is the $(1,1)$ element of $e^{tL_1^S}$. 
\end{lemma}
\begin{proof}
Define $A^S$ by
\begin{equation}\label{s2.16}
\begin{aligned}
A^S:=
\begin{pmatrix}
                       -b
                       & \eta_1^{-1} & \cdots     & \eta_{n-1}^{-1}           &    \eta_n^{-1}  \\

                             \eta_1^{-1}                           & -\eta_1^{-1} & \cdots     &    0               &            0         \\
            \cdots                                   & \cdots                   & \cdots     & \cdots             & \cdots               \\
                               \eta_{n-1}^{-1}                                & 0              & \cdots     & -\eta_{n-1}^{-1}     &            0         \\
                             \eta_n^{-1}                               & 0              & \cdots     &    0               &    -\eta_n^{-1}
\end{pmatrix}.
\end{aligned}
\end{equation}
Also, let $D^{\mu}={\rm diag}(\sqrt{2\mu_0},\sqrt{2\mu_1},\cdots,\sqrt{2\mu_n})$. Then, it is easy to see that 
\begin{equation}\label{s2.17}
\begin{aligned}
D^{\mu}A^SD^{\mu}=L_1^S.
\end{aligned}
\end{equation}
Now, observe that $g_{00}(t)$ is the $(1,1)$ element of $(2\mu_0)^{-1}D^{\mu}D^{\mu}e^{tL_1}$. Further, since $L_1=A^SD^{\mu}D^{\mu}$, we have 
\begin{equation}\label{s2.18}
\begin{aligned}
D^{\mu}D^{\mu}e^{tL_1}=D^{\mu}e^{tD^{\mu}A^SD^{\mu}}D^{\mu}=D^{\mu}e^{tL_1^S}D^{\mu}.
\end{aligned}
\end{equation}
Then, from \eqref{s2.18}, we have that $g_{00}(t)$ is the $(1,1)$ element of $e^{tL_1^S}$.

\end{proof}

\noindent
Based on Lemma \ref{s_lem2.1}, let $-\tau_j<0$ for $0\leq j\leq n$ be the eigenvalues of $L_1^S$ and the associated unit eigenvectors $v^j$ for $0\leq j\leq n$ such that
$L_1^Sv^j=-\tau_j v^j$ for $0\leq j\leq n$.
Also, let
$D^{\tau}={\rm diag}(-\tau_0,-\tau_1,\cdots,-\tau_n)$ and 
$V$ be a matrix with $n+1$-dimensional unit column vectors $v^j,\,0\le j\le n$ each of which has the i-th component $v_i^j$. Then we have
\begin{equation}\label{s2.19}
\begin{aligned}
L_1^S=VD^{\tau}V^t,
\end{aligned}
\end{equation}
which yields the following theorem.
\begin{Th}\label{s_thm2.3}
We have the following expression for $g_{00}(t)$:
\begin{equation}\label{s2.20}
\begin{aligned}
g_{00}(t)=\sum_{j=0}^n(v_0^j)^2e^{-t\tau_j}.
\end{aligned}
\end{equation}
\end{Th}
\begin{proof}
By using \eqref{s2.19}, we have that
\begin{equation}\label{s2.21}
\begin{aligned}
e^{tL_1^S}=e^{tVD^{\tau}V^t}=Ve^{tD^{\tau}}V^t.
\end{aligned}
\end{equation}
Combining Lemma \ref{s_lem2.2} and \eqref{s2.21}, we obtain \eqref{s2.20}.

\end{proof}

Next, in the same way as we did for $e^{t L_1^{\tilde I}}$, we compute the $(1,1)$ tensor block component of the tensor block matrix $e^{t L_0^{\tilde I}}$. For that let
\begin{equation}\label{s2.22}
L_0^S:=
\begin{pmatrix}
                       -b(3\lambda_0+2\mu_0)
                       & \eta_1^{-1}h_1  & \cdots    &        \eta_{n-1}^{-1}h_{n-1}     &    \eta_n^{-1}h_n   \\

                              \eta_1^{-1}h_1                             & -\eta_1^{-1}(3\lambda_1+2\mu_1) & \cdots                  &          0  &   0         \\
            \cdots                                   & \cdots                   & \cdots     & \cdots           & \cdots                 \\
                               \eta_{n-1}^{-1}h_{n-1}                                & 0              & \cdots          &      -\eta_{n-1}^{-1}(3\lambda_{n-1}+2\mu_{n-1}) &      0         \\
                              \eta_n^{-1}h_n                              & 0              & \cdots        & 0            &   -\eta_n^{-1}(3\lambda_n+2\mu_n)
\end{pmatrix},
\end{equation}
where $h_i=\sqrt{(3\lambda_0+2\mu_0)(3\lambda_i+2\mu_i)} $ for $1\leq i \leq n$.

\medskip
\noindent
Then, we can prove the following lemma in the same as we did for $L_1^S$.

\begin{lemma}\label{s_lem2.4}
$L_0^S$ is negative definite.
\end{lemma}

\medskip
\noindent
Further, by representing
$e^{tL_0(x)}$ as
\begin{equation}\label{s2.23}
e^{tL_0}=
\begin{pmatrix}
g^0_{00}(t)&g^0_{01}(t)&\cdots&g^0_{0n}(t)\\
\ldots& & &\ldots\\
g^0_{n0}(t)&g^0_{n1}(t)&\cdots&g^0_{nn}(t)
\end{pmatrix},
\end{equation}
we can prove the following lemma in the same way as we proved Lemma \ref{s_lem2.2}.
\begin{lemma}\label{s_lem2.5}
$g^0_{00}(t)$ is the $(1,1)$ element of $e^{tL_0^S}$. 
\end{lemma}

\medskip
\noindent
Based on Lemma \ref{s_lem2.4}, let $-\kappa_j<0$ for $0\leq j\leq n$ be the eigenvalues of $L_0^S$ with the associated unit eigenvectors $q^j$ for $0\leq j\leq n$ such that
$L_0^Sq^j=-\kappa_j q^j$ for $0\leq j\leq n$. Also, let
$D^{\kappa}={\rm diag}(-\kappa_0,-\kappa_1,\cdots,-\kappa_n)$ and $Q$ be a matrix with $n+1$-dimensional unit column vectors $q^j,\,0\le j\le n$ each of which has the i-th component $q_i^j$. Then we have
\begin{equation}\label{s2.24}
\begin{aligned}
L_0^S=QD^{\kappa}Q^{\mathfrak{t}},
\end{aligned}
\end{equation}
where $\mathfrak{t}$ denotes the transpose.
\medskip\noindent
Then, we can prove that this yields the following theorem in the same way as we proved Theorem \ref{s_thm2.3}

\begin{Th}\label{s_thm2.6}
We have the following expression for $g_{00}^0(t)$
\begin{equation}\label{s2.25}
\begin{aligned}
g^0_{00}(t)=\sum_{j=0}^n(q_0^j)^2e^{-t\kappa_j}.
\end{aligned}
\end{equation}
\end{Th}

Now, combining \eqref{s2.3}, \eqref{s2.20} and \eqref{s2.25}, we have 
\begin{equation}\label{s2.26}
\begin{aligned}
\psi=\psi^V+\psi^D=&\int_0^t g^0_{00}(t-s)3^{-1}(\nabla\cdot\dot{u})(s)\,ds\,I_3\\
&+\int_0^t g_{00}(t-s)\big(e[\dot{u}]-3^{-1}(\nabla\cdot\dot{u})I_3\big) (s)\,ds\\
=&\int_0^t \sum_{j=0}^n(q_0^j)^2e^{-(t-s)\kappa_j}3^{-1}(\nabla\cdot\dot{u})(s)\,ds\,I_3\\
&+\int_0^t \sum_{j=0}^n(v_0^j)^2e^{-(t-s)\tau_j}\big(e[\dot{u}]-3^{-1}(\nabla\cdot\dot{u})I_3\big) (s)\,ds.
\end{aligned}
\end{equation}
Further, from \eqref{s2.26}, we have the following explicit form of the stress-strain relation, which yields the explicit form of the relaxation tensor for the inhomogeneous isotropic EBM:
\begin{equation}\label{s2.27}
\begin{aligned}
\sigma=C_0\psi
=&(\lambda_0+\frac{2}{3}\mu_0)\int_0^t \sum_{j=0}^n(q_0^j)^2e^{-(t-s)\kappa_j}(\nabla\cdot\dot{u})(s)\,ds\,I_3\\
&+2\mu_0\int_0^t \sum_{j=0}^n(v_0^j)^2e^{-(t-s)\tau_j}\big(e[\dot{u}]-3^{-1}(\nabla\cdot\dot{u})I_3\big) (s)\,ds.
\end{aligned}
\end{equation}
By rewriting this using $e[u](0)=0$, and $\sum_{j=0}^n (q_0^j)^2=\sum_{j=0}^n (v_0^j)^2=1$, we have 
\begin{equation}\label{s2.28}
\begin{aligned}
\sigma
=&\big((\lambda_0+\frac{2}{3}\mu_0)(\nabla\cdot u)I_3\big)(t)+ 2\mu_0\big(e[u]-3^{-1}(\nabla\cdot u)I_3\big) (t)\\
&-(\lambda_0+\frac{2}{3}\mu_0)\int_0^t \sum_{j=0}^n\kappa_j(q_0^j)^2e^{-(t-s)\kappa_j}(\nabla\cdot u)(s)\,ds\,I_3\\
&-2\mu_0\int_0^t \sum_{j=0}^n\tau_j(v_0^j)^2e^{-(t-s)\tau_j}\big(e[u]-3^{-1}(\nabla\cdot u)I_3\big) (s)\,ds.
\end{aligned}
\end{equation}

\begin{remark}
\begin{itemize}

\noindent
\item[\rm{(i)}]
Up to scalar multiplication,  $(\nabla\cdot\dot u)I_3$ and $e[\dot u]-3^{-1}(\nabla\cdot\dot u)I_3$ in \eqref{s2.27} correspond to the volumetric part and deviatoric part, respectively. They are orthogonal. However, their tractions, the inner product of their contractions with the unit outer normal vector $\nu$ of $\partial\Omega$, are given as
$$
(\nabla\cdot\dot u)(e[\dot u],\nu\otimes\nu)-3^{-1}(\nabla\cdot\dot u)^2
$$
does not necessarily vanish, where $(\,\,,\,\,)$ denotes the inner product for the second order tensors. Hence, concerning the traction-free boundary condition, it is very hard to split them.
\item[\rm{(ii)}] In \eqref{s2.28}, $(\nabla\cdot u)I_3$ and $e[u]-3^{-1}(\nabla\cdot u)I_3$ appear jointly and separately outside and inside the integral, respectively. This is due to the new properties of the relaxation tensor that arise from the two aforementioned parts, and we present a new idea in the two forthcoming sections to generate the C-ev's.

\end{itemize}
\end{remark}

\section{Special eigenfunctions for the instantaneous
term}
\setcounter{equation}{0}
From this section, we prepare to generate the C-ev's for the FOE. Hence, in this section, we assume that the Earth $\Omega$ is an open ball with radius $R>0$ centered at the origin, and consider it as a homogeneous isotropic EBM. To begin with, we look for special radial functions that are eigenfunctions of the volumetric/deviatoric parts of the instantaneous part simultaneously, and satisfy the traction-free boundary condition for the instantaneous part. 

Let $Q_A$ and $Q_B$ be the operators defined as
\begin{equation}\label{s3.1}
\begin{aligned}
Q_Au:=(\lambda_0+\frac{2}{3}\mu_0)\nabla\cdot \big((\nabla\cdot u) I_3\big)=(\lambda_0+\frac{2}{3}\mu_0)\nabla (\nabla \cdot u)
\end{aligned}
\end{equation}
and
\begin{equation}\label{s3.2}
\begin{aligned}
Q_Bu:=2\mu_0\nabla\cdot \big(e[u]-3^{-1}(\nabla\cdot u) I_3\big).
\end{aligned}
\end{equation}
Also, let $H(p)$ be the Hessian of $p$ given as $H(p)=(\partial^2_{x_i,x_j}p)_{i,j=1}^3$. 
\begin{lemma}\label{s_lem3.1}
Let $u$ satisfy
\begin{equation}\label{s3.3}
\begin{array}{l}
\begin{cases}
Q_Au=-(\lambda_0+\frac{2}{3}\mu_0)k_a\, u\quad &\mbox{\rm in}\quad \Omega,\\
Q_Bu=-\frac{4}{3}\mu_0k_a\, u\quad &\mbox{\rm in}\quad \Omega,\\
\big(\lambda_0(\nabla \cdot u) I_3+2\mu_0e[u]\big)\nu=0      & {\rm on}\quad \partial\Omega
\end{cases}
\end{array} 
\end{equation}
for some constant $k_a>0$.
Then, $p:=\nabla\cdot u$ satisfies
\begin{equation}\label{s3.4}
\begin{array}{l}
\begin{cases}
\Delta p=-k_a\, p&\mbox{\rm in}\quad \Omega,\\
\big(\lambda_0p\tilde I-2\mu_0k_a^{-1}H(p)\big)\nu=0    & {\rm on}\quad \partial\Omega.
\end{cases}
\end{array}
\end{equation}
\end{lemma}
\begin{proof}
Direct computations give that
\begin{equation}\label{s3.5}
\begin{aligned}
 (\lambda_0+\frac{2}{3}\mu_0)\Delta p=\nabla \cdot (Q_A u)=-(\lambda_0+\frac{2}{3}\mu_0)k_a\, p
\end{aligned}
\end{equation}
and
\begin{equation}\label{s3.6}
\begin{aligned}
 2\mu_0(1-3^{-1})\Delta p=\nabla \cdot (Q_B(x)u)= -\frac{4}{3}\mu_0k_a\, p.
\end{aligned}
\end{equation}
Adding \eqref{s3.5} and \eqref{s3.6}, we have the first equation of \eqref{s3.4}. Using the first equation of \eqref{s3.3}, we have  
$$\nabla p=\nabla (\nabla \cdot u)=-k_au,\,\,\text{and hence}\,\,e[u]=-k_a^{-1}H(p).$$
Then, plugging these into the third equation of \eqref{s3.3}, we have the second equation of \eqref{s3.4}.
\end{proof}
  
\medskip\noindent
On the other hand, the converse of Lemma \ref{s_lem3.1} is also true. More precisely, we have the following lemma.

\begin{lemma}\label{s_lem3.2}
Let  $p$ satisfy
\begin{equation}\label{s3.7}
\begin{array}{l}
\begin{cases}
\Delta p=-k_b\, p&\mbox{\rm in}\quad \Omega,\\
\big(\lambda_0p I_3 -2\mu_0k_b^{-1}H(p)\big)\nu=0    & {\rm on}\quad \partial\Omega,
\end{cases}
\end{array}
\end{equation}
where $k_b>0$ is a constant. 
Then, $u:=-k_b^{-1}\nabla p$ satisfies
\begin{equation}\label{s3.8}
\begin{array}{l}
\begin{cases}
Q_Au=-(\lambda_0+\frac{2}{3}\mu_0)k_b\, u\quad &\mbox{\rm in}\quad \Omega,\\
Q_Bu=-\frac{4}{3}\mu_0k_b\, u\quad &\mbox{\rm in}\quad \Omega,\\
\big(\lambda_0(\nabla \cdot u) I_3+2\mu_0e[u]\big)\nu=0     & {\rm on}\quad \partial\Omega.
\end{cases}
\end{array} 
\end{equation}
Further, we have $p=\nabla\cdot u$, which implies the equivalence of \eqref{s3.3} and \eqref{s3.4}.
\end{lemma}

\begin{proof}
Direct computations give that
\begin{equation}\label{s3.9}
\begin{array}{ll}
Q_A(x)u&=(\lambda_0+\frac{2}{3}\mu_0)\nabla (\nabla\cdot u) 
=-(\lambda_0+\frac{2}{3}\mu_0)k_b^{-1}\nabla \Delta p\\
&=(\lambda_0+\frac{2}{3}\mu_0)\nabla p
=-(\lambda_0+\frac{2}{3}\mu_0)k_b\,u
\end{array}
\end{equation}
and
\begin{equation}\label{s3.10}
\begin{array}{ll}
Q_B(x)u&=\mu_0\nabla (\nabla\cdot u)+\mu_0 \Delta u-\frac{2}{3}\mu_0 \nabla (\nabla\cdot u)\\
&=-\mu_0k_b^{-1}\nabla \Delta p-\mu_0 k_b^{-1}\nabla \Delta p+\frac{2}{3}\mu_0 k_b^{-1}\nabla \Delta p\\
&=\frac{4}{3}\mu_0\nabla p
=-\frac{4}{3}\mu_0k_b\,u.
\end{array}
\end{equation}
Also, we have
$$-k_b\nabla \cdot u=\nabla \cdot \nabla p=\Delta p=-k_b p$$
which implies $p=\nabla \cdot u$. Then, we can finish the proof by showing the third equation of \eqref{s3.8} by direct computations.
\end{proof}

Now, we seek $p$ as a radial function $p(x)=p_1(r)$ with $r=|x|$ to satisfy \eqref{s3.7}. By direct computations, we have
$$H(p)\nu=\nabla p_1'(r)=p_1''(r)\nu.$$
Thus, \eqref{s3.7} is transformed to 
\begin{equation}\label{s3.11}
\begin{array}{l}
\begin{cases}
r^2 p_1''(r)+2rp_1'(r)=-k_br^2 p_1(r)&\mbox{\rm for}\quad 0\leq r< R,\\
\lambda_0p_1(r)-2\mu_0k_b^{-1}p_1''(r)=0    & {\rm for}\quad r=R.
\end{cases}
\end{array}
\end{equation}
Using the first equation of \eqref{s3.11}, we have
\begin{equation}\label{s3.12}
\begin{aligned}
\lambda_0p_1(r)-2\mu_0k_b^{-1}p_1''(r)&=\lambda_0p_1(r)-2\mu_0k_b^{-1}(-k_bp_1(r)-2r^{-1}p_1'(r))\\
&=(\lambda_0+2\mu_0)p_1(r)+4\mu_0k_b^{-1}r^{-1}p_1'(r)=0.
\end{aligned}
\end{equation}
Hence, using \eqref{s3.12}, \eqref{s3.11} becomes
\begin{equation}\label{s3.13}
\begin{array}{l}
\begin{cases}
r^2 p_1''(r)+2rp_1'(r)+k_br^2 p_1(r)=0,&\mbox{\rm for}\quad 0\leq r< R,\\
(\lambda_0+2\mu_0)p_1(r)+4\mu_0k_b^{-1}r^{-1}p_1'(r)=0    & {\rm for}\quad r=R.
\end{cases}
\end{array}
\end{equation}
By putting $\sqrt{k_b} r=\eta$ and $p_2(\eta)=p_1(\eta/\sqrt{k_b})$, we rewrite \eqref{s3.13} as
\begin{equation}\label{s3.14}
\begin{array}{l}
\begin{cases}
\eta^2 p_2''(\eta)+2\eta p_2'(\eta)+\eta^2 p_2(\eta)=0,&\mbox{\rm for}\quad 0\leq \eta< \sqrt{k_b}R,\\
(\lambda_0+2\mu_0)p_2(\eta)+4\mu_0\eta^{-1}p_2'(\eta)=0    & {\rm for}\quad \eta=\sqrt{k_b}R.
\end{cases}
\end{array}
\end{equation}

Next, consider the spherical Bessel function $J(\eta)$ of order $0$ given as
\begin{equation}\label{s3.15}
\begin{array}{l}
J(\eta):=\frac{\sin \eta}{\eta}=\displaystyle\sum_{j=0}^\infty\frac{(-1)^j}{(2j+1)!}\eta^{2j},
\end{array}
\end{equation}
which is analytic in ${\mathbb R}$ (see \cite{Kor02}). Then, $J(\eta)$ satisfies
\begin{equation}\label{s3.16}
\begin{aligned}
\eta^2 J''(\eta)+2\eta J'(\eta)+\eta^2 J(\eta)=0\,\,\,\text{for}\,\,\,\eta\in{\mathbb R}.
\end{aligned}
\end{equation}
Comparing the first equation of \eqref{s3.14} and \eqref{s3.16}, we consider $p_2(\eta)$ as $J(\eta)$.

Next, concerning the second equation of \eqref{s3.14}, we search for positive roots of  $$(\lambda_0+2\mu_0)J(\eta)+4\mu_0\eta^{-1}J'(\eta)=0,$$
which is equivalent to search for positive roots of the following equation
\begin{equation}\label{s3.17}
\begin{aligned}
f(\eta):
=&\big((\lambda_0+2\mu_0)\eta^2-4\mu_0\big)\sin\eta+4\mu_0\eta\cos\eta=0.
\end{aligned}
\end{equation}

\medskip\noindent
Then, we have the following lemma.
\begin{lemma}\label{s_lem3.3}
There exists a unique $r_\ell\in((\ell-1)\pi, \ell \pi)$ such that $f(r_\ell)=0$ for each  $\ell=1,2,\cdots$. In fact, we have
$r_l\in((l-\frac{1}{2})\pi, l\pi)$ for each  $l=1,2,\cdots$. 
\end{lemma}
\begin{proof}
By direct computations, we have
\begin{equation}\label{s3.18}
\begin{aligned}
f'(\eta)=(\lambda_0+2\mu_0)\eta^2\cos\eta+2\lambda_0\eta\sin\eta
\end{aligned}
\end{equation}
and
\begin{equation}\label{s3.19}
\begin{aligned}
f''(\eta)=4(\lambda_0+\mu_0)\cos\eta+\big(2\lambda_0-(\lambda_0+2\mu_0)\eta^2\big)\sin\eta.
\end{aligned}
\end{equation}
We divide $\ell$ into two cases.

\medskip
\noindent
Case 1: $\ell=2m-1$ with a positive integer $m$.

From \eqref{s3.18}, we have that 
\begin{equation}\label{s3.20}
\begin{array}{l}
\begin{cases}
f'(\eta)>0,&\mbox{\rm if}\quad \eta\in\big((2m-2)\pi, (2m-\frac{3}{2})\pi\big),\\
f'(\eta)<0    & {\rm if}\quad \eta=(2m-1)\pi.
\end{cases}
\end{array}
\end{equation}
Since
\begin{equation*}
\begin{aligned}
2\lambda_0-(\lambda_0+2\mu_0)\eta^2<0\,\,\text{for}\,\,\eta\in[(2m-\frac{3}{2})\pi,(2m-1)\pi]
\end{aligned}
\end{equation*}
if $\eta\in[(2m-\frac{3}{2})\pi,(2m-1)\pi]$. Thus, we obtain that
\begin{equation}\label{s3.21}
\begin{aligned}
f''(\eta)<0
\end{aligned}
\end{equation}
if $\eta\in\big((2m-\frac{3}{2})\pi,(2m-1)\pi)\big)$. 
By \eqref{s3.20} and \eqref{s3.21}, there exists a unique $\eta_{2m-1}\in \big((2m-\frac{3}{2})\pi,(2m-1)\pi\big)$ such that $f'(\eta_{2m-1})=0$. We also have that
\begin{equation}\label{s3.22}
\begin{array}{l}
\begin{cases}
f'(\eta)>0,&\mbox{\rm if}\quad \eta\in((2m-2)\pi, \eta_{2m-1}),\\
f'(\eta)<0    & {\rm if}\quad \eta\in(\eta_{2m-1}, (2m-1)\pi).
\end{cases}
\end{array}
\end{equation}
Hence, by $f((2m-2)\pi)> 0$, $f((2m-1)\pi)< 0$ and \eqref{s3.22}, there exists a unique $r_{2m-1}\in (\eta_{2m-1},(2m-1)\pi)$ such that $f(r_{2m-1})=0$.

\medskip\noindent
Case 2: $\ell=2m$ with a positive integer $m$.

From \eqref{s3.18}, we have 
\begin{equation}\label{s3.23}
\begin{array}{l}
\begin{cases}
f'(\eta)<0,&\mbox{\rm if}\quad \eta\in\big((2m-1)\pi, (2m-\frac{1}{2})\pi\big),\\
f'(\eta)>0    & {\rm if}\quad \eta=2m\pi.
\end{cases}
\end{array}
\end{equation}
Since
\begin{equation*}
\begin{aligned}
2\lambda_0-(\lambda_0+2\mu_0)\eta^2<0\,\,\,\,\text{for}\,\,\,\,
\eta\in[(2m-\frac{1}{2})\pi,2m\pi],
\end{aligned}
\end{equation*}
we have
\begin{equation}\label{s3.24}
\begin{aligned}
f''(\eta)>0\,\,\,\,\text{for}\,\,\,\,\eta\in[(2m-\frac{1}{2})\pi,2m\pi)].
\end{aligned}
\end{equation}
By \eqref{s3.23} and \eqref{s3.24}, there exists a unique $\eta_{2m}\in \big((2m-\frac{1}{2})\pi,2m\pi\big)$ such that $f'(\eta_{2m})=0$. We also have
\begin{equation}\label{s3.25}
\begin{array}{l}
\begin{cases}
f'(\eta)<0,&\mbox{\rm if}\quad \eta\in((2m-1)\pi, \eta_{2m}),\\
f'(\eta)>0    & {\rm if}\quad \eta\in(\eta_{2m}, 2m\pi).
\end{cases}
\end{array}
\end{equation}
Hence, by $f((2m-1)\pi)< 0$, $f(2m\pi)> 0$ and \eqref{s3.25}, there exists a unique $r_{2m}\in (\eta_{2m},2m\pi)$ such that $f(r_{2m})=0$.

\end{proof}

\medskip\noindent
Based on the previous lemma, the following lemma states that we can completely solve \eqref{s3.13} by taking $k_b=\frac{r_\ell^2}{R^2}$ for  $\ell=1,2,\cdots$.
\begin{lemma}\label{s_lem3.4}
For $\ell=1,2,\cdots$, let $y_\ell(r)=J(\frac{r_\ell}{R}r)$. Then $y_\ell$ satisfies
\begin{equation}\label{s3.26}
\begin{array}{l}
\begin{cases}
r^2 y_\ell''(r)+2ry_\ell'(r)=-k_br^2 y_\ell(r)&\mbox{\rm for}\quad 0\leq r< R,\\
(\lambda_0+2\mu_0)y_\ell(r)+4\mu_0k_b^{-1}r^{-1}y_\ell'(r)=0     & {\rm for}\quad r=R.
\end{cases}
\end{array}
\end{equation}
\end{lemma}
\begin{proof}
The first equation of \eqref{s3.26} follows from \eqref{s3.16}. As for the second equation of \eqref{s3.26}, it follows from $f(r_\ell)=0$ for $\ell=1,2,\cdots$.
\end{proof}
Thus, we have solved \eqref{s3.13} to get $p(x)=p_1(r)$. Then, we can get $u=-k_b^{-1} \nabla p$   
satisfying \eqref{s3.8}. 
Also, by direct computations, we have that
\begin{equation}\label{s3.27}
\begin{array}{ll}
u(x)&=-k_b^{-1}  y'_\ell(r) \frac{x}{r}=-\frac{R^2}{r_\ell^2}\big(r\cos(\frac{r_\ell}{R}r)-\frac{R}{r_\ell}\sin (\frac{r_\ell}{R}r)\big)\frac{x}{r^3}\\
&=-l_b^{-3}\{\displaystyle\sum_{i=0}^\infty\big((2i)!\big)^{-1}(-1)^i r(l_b r)^{2i+1}+\displaystyle\sum_{j=1}^\infty\big((2j-1)!\big)^{-1} (-1)^j(l_br)^{2j-1}\}\frac{x}{r^3}
\end{array}
\end{equation}
for $x\in\overline\Omega$, where $r=|x|$, $l_b=\sqrt{k_b}$ and $r_\ell$, $\ell=1,2,\cdots$ satisfy 
\begin{equation}\label{s3.28}
\begin{aligned}
f(r_\ell)=(\lambda_0+2\mu_0)r^2_\ell\sin r_\ell+4\mu_0(r_\ell\cos r_\ell-\sin r_\ell)=0.
\end{aligned}
\end{equation}
Note here that, since the first terms of the summations in \eqref{s3.27} cancel out, $u(x)$ is analytic on $\overline\Omega$. Therefore, $u(x)$ of \eqref{s3.27} satisfies
\begin{equation}\label{s3.29}
\begin{array}{l}
\begin{cases}
Q_Au=-(\lambda_0+\frac{2}{3}\mu_0)\frac{r_\ell^2}{R^2}\, u\quad &\mbox{\rm in}\quad \Omega,\\
Q_Bu=-\frac{4}{3}\mu_0\frac{r_\ell^2}{R^2}\, u\quad &\mbox{\rm in}\quad \Omega,\\
\big(\lambda_0(\nabla \cdot u)\tilde I+2\mu_0e[u]\big)\nu=0     & {\rm on}\quad \partial \Omega
\end{cases}
\end{array} 
\end{equation}
for $\ell=1,2\cdots$.

\section{C-ev's for the FOE }
\setcounter{equation}{0}

In this section, we consider cluster eigenvalues for the FOE using the special eigenfunctions for the instantaneous part given in the previous section. First of all, from \eqref{s2.28}, \eqref{s3.1} and \eqref{s3.2}, the BVS becomes
\begin{equation}\label{s4.1}
\begin{array}{rcl}
\partial_t^2 u(t) &=&Q_Au(t)+Q_Bu(t)-\displaystyle\sum_{j=0}^n \kappa_j(q_0^j)^2\int_0^t e^{-\kappa_j(t-s)}Q_Au(s)\,ds\\
&&-\displaystyle\sum_{j=0}^n \tau_j(v_0^j)^2\int_0^t e^{-\tau_j(t-s)}Q_Bu(s)\,ds\qquad\text{in $(0,\infty)\times\Omega$}.
\end{array}
\end{equation}
Here and hereafter, note that likewise before, we suppressed the variable $x$.
Without loss of the generality, we assume that
\begin{equation}\label{s4.2}
\begin{aligned}
\tau_0<\tau_1<\cdots<\tau_n<\kappa_0<\kappa_1<\cdots<\kappa_n.
\end{aligned}
\end{equation}
For simplicity, we introduce the following shorthand notation 
\begin{equation}\label{s4.3}
\left\{
\begin{array}{ll}\beta_{j+1}=\tau_j,\, \alpha_{j+1}=\frac{4}{3}\mu_0\tau_j(v_0^j)^2,\,  w_{j+1}:=(\frac{4}{3}\mu_0)^{-1}\int_0^t e^{-\tau_j(t-s)}Q_Bu(s)\,ds,\,\beta_{j+n+2}=\kappa_j,
 \\
\\
 \alpha_{j+n+2}=(\lambda_0+\frac{2}{3}\mu_0)\kappa_j(q_0^j)^2, \,\,w_{j+n+2}:=(\lambda_0+\frac{2}{3}\mu_0)^{-1}\int_0^t e^{-\kappa_j(t-s)}Q_Au(s)\,ds
\end{array}
\right.
\end{equation}  
for $0\leq j\leq n$. Then, \eqref{s4.1} takes the form
\begin{equation}\label{s4.4}
\begin{aligned}
\partial_t^2 u=&Q_A u +Q_B u-  \sum_{j=1}^{2n+2} \alpha_jw_j\\
=&Q_A u +Q_B u- \sum_{j=1}^{n+1} \alpha_j\int_0^t e^{-\beta_j(t-s)}\nabla\cdot \big(\frac{3}{2}e[u]-\frac{1}{2}(\nabla\cdot u) I_3\big)(s)\,ds\\
&-\sum_{j=n+2}^{2n+2} \alpha_j\int_0^t e^{-\beta_j(t-s)}\nabla\cdot \big((\nabla\cdot u) I_3\big)(s)\,ds.
\end{aligned}
\end{equation}
Here, we note a little ahead of time that the inverse problem in the next section, in which C-ev's are used as spectral data, is to recover the coefficient $(\lambda_0+2\mu_0)k_b$ of $Q_A u+Q_B u$ with $u$ satisfying \eqref{s3.8} obtained at the end of Section 3, and $\alpha_j$'s, $\beta_j$'s.

Now, by introducing a new dependent variable $v=u'$, we write \eqref{s4.1} in the form of a system of partial differential equations that are of first order in time given as
\begin{equation}\label{s4.5}
\left\{
\begin{array}{rcl}
u'&=&v,\\
v'&=&Q_A u +Q_B u-   \sum_{i=1}^{n+2} \alpha_iw_i,\\
w_j'&=&(\frac{4}{3}\mu_0)^{-1}Q_B u -\beta_jw_j\,\,\,\text{for}\,\,\, 1\leq j \leq n+1\\
w_j'&=&(\lambda_0+\frac{2}{3}\mu_0)^{-1}Q_A u -\beta_jw_j\,\,\,\text{for}\,\,\, n+2\leq j \leq 2n+2.
\end{array}
\right.
\end{equation}
We refer to this system as an augmented system. The matrix-vector representation of this system is \begin{equation}\label{s4.6}
\begin{aligned}
U'=A_{2n+4}U,
\end{aligned}
\end{equation}
where $A_{2n+4}$ is a $2n+4$ square matrix of the second-order partial differential operators in $x$ and $U=(u,v,w_1,\cdots,w_{2n+2})^{\mathfrak{t}}$. More precisely, we consider $A_{2n+4}$ as a densely defined closed operator on the $2n+4$ number of products of $L^2(\Omega)$, denoted by $L^2(\Omega)^{2n+4}$, with the domain
\begin{equation}\label{s4.7}
D(A_{2n+4}):=\left\{
\begin{array}{ll}
U=(u,v,w_1,\cdots,w_N)^{\mathfrak{t}}\in H^2(\Omega)\times H^1(\Omega)\times L^2(\Omega)^{2n+2}:\\
\qquad\big(\lambda_0(\nabla \cdot u)\tilde I+2\mu_0e[u]\big)\nu=0  \quad    {\rm on}\quad \partial \Omega,\,t>0
\end{array}
\right\},
\end{equation}
where $H^s(\Omega),\,s=1,2$ are the $L^2(\Omega)$-based  Sobolev spaces of order $s=1,2$. Then, we have an eigenvalue problem given as
\begin{equation}\label{s4.8}
\begin{aligned}
z U=A_{2n+4}\,U.
\end{aligned}
\end{equation}
The cluster eigenvalues problems associated to \eqref{s4.8} that we are interested in are as follows. Plugging the solution $u$ depending on $\ell$ of \eqref{s3.29} into \eqref{s4.8}, we consider the following reduced eigenvalues problem given as
\begin{equation}\label{s4.9}
\begin{aligned}
z\,U^\ell=A_{2n+4}^\ell\,U^\ell.
\end{aligned}
\end{equation}
More precisely, for each $\ell\in{\mathbb N}$, we look for an eigenvalue $z\not=-\beta_j,\,1\le j\le 2n+2$ of \eqref{s4.9} with the associated eigenfunction $U^\ell=(u,v,w_1,\cdots,w_{2n+2})^{\mathfrak{t}}$ of \eqref{s4.8}. A detail form of \eqref{s4.9} is given as
\begin{equation}\label{s4.10}
\left\{
\begin{array}{ll}
zu=v,\\
zv=-(\lambda_0+2\mu_0)\frac{r_\ell^2}{R^2} u-  \sum_{i=1}^{n+2} \alpha_iw_i,\\
zw_j=-\frac{r_\ell^2}{R^2} u -\beta_jw_j, \qquad\quad&\ 1\leq j \leq 2n+2
\end{array}
\right.
\end{equation}
and the form of matrix $A^\ell_{2n+4}$ can be read from \eqref{s4.8}. It is important to note here that $u$ satisfying \eqref{s4.10} is an eigenfunction of the q-EBM with eigenvalue $z^2$ satisfying the traction-free boundary condition for the q-EBM. This is because $u$ satisfies the instantaneous elastic equation with a traction-free boundary condition and this boundary condition is invariant under adding any zero-order terms with real coefficients. 

From \eqref{s4.10},  we have the following
\begin{equation}\label{s4.11}
\left\{
\begin{array}{ll}
z^2u+(\lambda_0+2\mu_0)\frac{r_\ell^2}{R^2} u+  \sum_{i=1}^{n+2} \alpha_iw_i=0,\\
(z+\beta_j)w_j=-\frac{r_\ell^2}{R^2} u, \qquad\quad&\ 1\leq j \leq 2n+2.
\end{array}
\right.
\end{equation}
Multiplying $\Pi_{1\leq j\leq 2n+2}(z+\beta_j)$ to the first equation of \eqref{s4.11} and using the second and third equations of \eqref{s4.11}, we obtain
\begin{equation}\label{s4.12}
\begin{aligned}
(z^2+(\lambda_0+2\mu_0)\frac{r_\ell^2}{R^2})\Pi_{1\leq j\leq 2n+2}(z+\beta_j) - \frac{r_\ell^2}{R^2} \sum_{i=1}^{2n+2} \alpha_i\Pi_{1\leq j\leq 2n+2,j\neq i}(z+\beta_j)=0,
\end{aligned}
\end{equation}
which is nothing but
\begin{equation}\label{s4.13}
\begin{aligned}
((\lambda_0+2\mu_0)+\frac{R^2}{r_l^2}z^2)\Pi_{1\leq j\leq 2n+2}(z+\beta_j) -  \sum_{i=1}^{2n+2} \alpha_i\Pi_{1\leq j\leq 2n+2,j\neq i}(z+\beta_j)=0.
\end{aligned}
\end{equation}
From the arguments deriving \eqref{s4.12}, it is clear that for each $\ell\in{\mathbb N}$, the roots of \eqref{s4.12} are the eigenvalues of $A_{2n+4}^\ell$ and the associated eigenvectors can be obtained from \eqref{s3.27} and \eqref{s4.5}.
We define the characteristic polynomial $P^\ell_{2n+4}(z)$ associated to $\partial_t-A_{2n+4}^\ell$ and its limit polynomial $P_{2n+2}(z)$ as $\ell\rightarrow\infty$ by
\begin{equation}\label{s4.14}
\begin{aligned}
P^\ell_{2n+4}(z):=&((\lambda_0+2\mu_0)+\frac{R^2}{r_l^2}z^2)\Pi_{1\leq j\leq 2n+2}(z+\beta_j) -  \sum_{i=1}^{2n+2} \alpha_i\Pi_{1\leq j\leq 2n+2,j\neq i}(z+\beta_j)
\end{aligned}
\end{equation}
and
\begin{equation}\label{s4.15}
\begin{aligned}
P_{2n+2}(z):=&(\lambda_0+2\mu_0)\Pi_{1\leq j\leq 2n+2}(z+\beta_j) - \sum_{i=1}^{2n+2} \alpha_i\Pi_{1\leq j\leq 2n+2,j\neq i}(z+\beta_j)
\end{aligned}
\end{equation}
respectively. We note that $P_{2n+2}(z)$ is the charactersite polynomial associated to  
\begin{equation}\label{s4.16}
\left\{
\begin{array}{rcl}
0&=&Q_A u +Q_B u-  \sum_{j=1}^{2n+2} \alpha_jw_j,\\
w_j'&=&(\frac{4}{3}\mu_0)^{-1}Q_B u -\beta_jw_j\,\,\,\text{for}\,\,\, 1\leq j \leq n+1\\
w_j'&=&(\lambda_0+\frac{2}{3}\mu_0)^{-1}Q_A u -\beta_jw_j\,\,\,\text{for}\,\,\, n+2\leq j \leq 2n+2.
\end{array}
\right.
\end{equation}
This can be easily shown in the same way as we derived $P^\ell_{2n+4}(z)$.

\section{Structure of C-ev's and recovering q-EBM from C-ev's}
\setcounter{equation}{0}

In \cite{CDDLN1} and \cite{CDDLN2}, the roots of the characteristic polynomial equation \eqref{s5.1} given below were used to recover the relaxation tensor of the EBM describing its glassy state by the Prony series. Usually, the stretched exponential function is used instead of its approximation, the Prony series (see \cite{CDDLN1, CDDLN2} and references there in).
\begin{equation}\label{s5.1}
\begin{aligned}
P_N^k(\lambda):=(D+\frac{\lambda^2}{(2k-1)^2})\Pi_{1\leq j\leq N}(\lambda+r_j) -  \sum_{i=1}^N b_i\Pi_{1\leq j\leq N,j\neq i}(\lambda+r_j) =0.
\end{aligned}
\end{equation}
Comparing \eqref{s4.14} with \eqref{s5.1}, the notational correspondence between these two characteristic equations is given as
$$\lambda=z,\,N=2n+2,\,D=\lambda_0+2\mu_0,\, (2k-1)^2=\frac{r_\ell^2}{R^2},\,r_j=\beta_j,\,b_j=\alpha_j.$$
Here, the notation $r_\ell$ and $r_j$ is a bit confusing, but they can be distinguished by being aware that the left-hand side and right-hand side of each equation come from \eqref{s5.1} and what we have introduced before, respectively. 
Hence, we can literally transform Lemma 3.2 of \cite{CDDLN1} on the properties of the roots of \eqref{s5.1} to those of \eqref{s4.14}. Thus, we have the following lemma on the structure of C-ev's.
\begin{lemma}\label{s_lem5.1}
There exists at least $2n+2$ real roots $a^\ell_j$, $1\leq j\leq 2n+2$ of  $P^\ell_{2n+4}(z)=0$ such that 
\begin{equation}\label{s5.2}
\begin{aligned}
-\beta_{2n+2}<a^\ell_{2n+2}<-\beta_{2n+1}<a^\ell_{2n+1}<\cdots<-\beta_{2}<a^\ell_2<-\beta_{1}<a^\ell_1
\end{aligned}
\end{equation}
and
\begin{equation}
\label{s5.3}
 \sum_{i=1}^{2n+2} \frac{\alpha_i}{a^\ell_j+\beta_i} =(\lambda_0+2\mu_0)+\frac{R^2}{r_l^2}(a^\ell_j)^2,\,\,1\le j\le 2n+2.
\end{equation}
The other two roots are contained in the set $B^D_-\cup \{(-\beta_{2n+2},a^\ell_1)\} $, where $B^D_-:=\{c+id :\frac{-\beta_{2n+2}-a^\ell_1}{2}<c<\frac{-\beta_{1}-a^\ell_1}{2}\} $.
\end{lemma}

\medskip
Now, let us formulate an inverse problem recovering the relaxation tensor of the q-EBM from the C-ev's of the FOE as follows.

\medskip
\noindent
{\bf Inverse problem:} Recover $\alpha_j,\, \beta_j$ with $1\le j\le 2n+2$ of \eqref{s4.4} and $\lambda_0+2\mu_0$ by knowing sets of the clustered eigenvalues of \eqref{s4.14}.

\medskip
Then, we have the following answer to the above inverse problem.
\begin{Th}\label{s_thm5.2}
By knowing two clusters of eigenvalues associated with $\ell=\ell_1,\,\ell_2\in{\mathbb N}$, we can recover $\alpha_j,\, \beta_j$ with $1\le j\le 2n+2$ of \eqref{s4.4} and $\lambda_0+2\mu_0$. 
\end{Th}
\begin{proof}
In Theorem 1 of \cite{CDDLN2},  we recover $D,\,b_j,\,r_j$, $1\leq j\leq N$ of \eqref{s5.1} by using two clusters of eigenvalues. In the same way, we can recover $\lambda_0+2\mu_0$, $\alpha_j$, $\beta_j$, $1\leq j\leq 2n+2$.
\end{proof}

\begin{remark}
We can find suitable $\mu_0$ which depends on the order of \eqref{s4.2} and then get the associated $\lambda_0$, and $(v_0^j)^2$, $(q_0^j)^2$ for $0\leq j \leq n$ in \eqref{s4.1}. In fact, from the second equation and the fifth equation in \eqref{s4.3}, we have 
\begin{equation}\label{s5.4}
\begin{aligned}
\frac{\alpha_{j+1}}{\beta_{j+1}}=\frac{4\mu_0}{3}(v_0^j)^2,\quad \frac{\alpha_{j+n+2}}{\beta_{j+n+2}}=(\lambda_0+\frac{2}{3}\mu_0)(q_0^j)^2,\qquad 0\leq j \leq n.
\end{aligned}
\end{equation}
Since $ \sum_{j=0}^{n}(v_0^j)^2=1 $ and $ \sum_{j=0}^{n}(q_0^j)^2=1 $, we obtain from \eqref{s5.4} that
\begin{equation}\label{s5.5}
\begin{aligned}
\mu_0=\frac{3}{4}\sum_{j=1}^{n+1}\frac{\alpha_{j}}{\beta_{j}},\quad \lambda_0+2\mu_0=\sum_{j=1}^{2n+2}\frac{\alpha_{j}}{\beta_{j}}.
\end{aligned}
\end{equation}
Note that, by the first equation in \eqref{s5.5}, $\mu_0$ depends on the ordering \eqref{s4.2}.
By using \eqref{s5.5}, we can determine $\lambda_0$ by
\begin{equation}\label{s5.6}
\begin{aligned}
\lambda_0=\sum_{j=1}^{2n+2}\frac{\alpha_{j}}{\beta_{j}}-\frac{3}{2}\sum_{j=1}^{n+1}\frac{\alpha_{j}}{\beta_{j}}=\sum_{j=n+2}^{2n+2}\frac{\alpha_{j}}{\beta_{j}}-\frac{1}{2}\sum_{j=1}^{n+1}\frac{\alpha_{j}}{\beta_{j}}.
\end{aligned}
\end{equation}
From \eqref{s5.4}, we can derive $(v_0^j)^2$ and $(q_0^j)^2$ for $0\leq j \leq n$, respectively. 
\end{remark}

\section{Conclusions and discussions }
Being stimulated by the statement given in \cite{YP} on the importance of the EBM for analyzing an inelastic effect coming from the upper mantle to normal modes of the Earth, we analyzed the cluster eigenvalues of the Earth using a three-dimensional homogeneous isotropic EBM. When we started this research, we first noticed that an explicit form of the relaxation tensor was missing for the three-dimensional isotropic generalized Burgers model (EBM') not necessarily homogeneous. Even though in our previous research, we obtained the relaxation tensor for a more general EBM, including the anisotropic case, it lacked explicitness. It was not as easy as we thought to fill this gap. To overcome this difficulty, we found that the volumetric and deviatoric decompositions worked very well. Thus, we developed a new method to derive an explicit form of the relaxation tensor of the aforementioned EBM', which is a fundamental contribution to the people studying EBM'.

On the other hand, for generating the C-ev's of the EBM, we had to pay the price of using the aforementioned decompositions. That is, for the instantaneous part of the q-EBM, we had to find a $u$ that is an eigenfunction for both of the operators $Q_A$ and $Q_B$ with a traction-free boundary condition, simultaneously (see Section 3). We refer to this eigenvalue problem as the simultaneous eigenvalue problem. Comparing the set of eigenvalues for the instantaneous part of the q-EBM with the traction-free boundary condition and that of the simultaneous eigenvalue problem, the latter one is much smaller than the former one. This is a new feature of the C-ev's for the FOE, showing its disadvantage for seeking eigenvalues of the q-EBM. Nevertheless, the C-ev's have a very nice structure which enables to recover the relaxation tensor of the homogeneous EBM. Based on these, we speculate that the C-ev’s have the potential to become popular
in the fields of geophysics and material science.

As for some future works, we provide numerical verifications of the structure of the mentioned C-ev's and the mentioned recovery in our next paper. Also, we are planning to provide an explicit form of the relaxation tensor for the mentioned more general EBM by using the joint spectral decomposition of elasticity tensors in the EBM by assuming that they are commutative. Furthermore, we are planning to use eigenfunctions of the C-ev’s as correlation functions to numerically find the C-ev’s.

\begin{figure*}[htb]
  \centering
\fbox{
\includegraphics[width=0.6\textwidth]{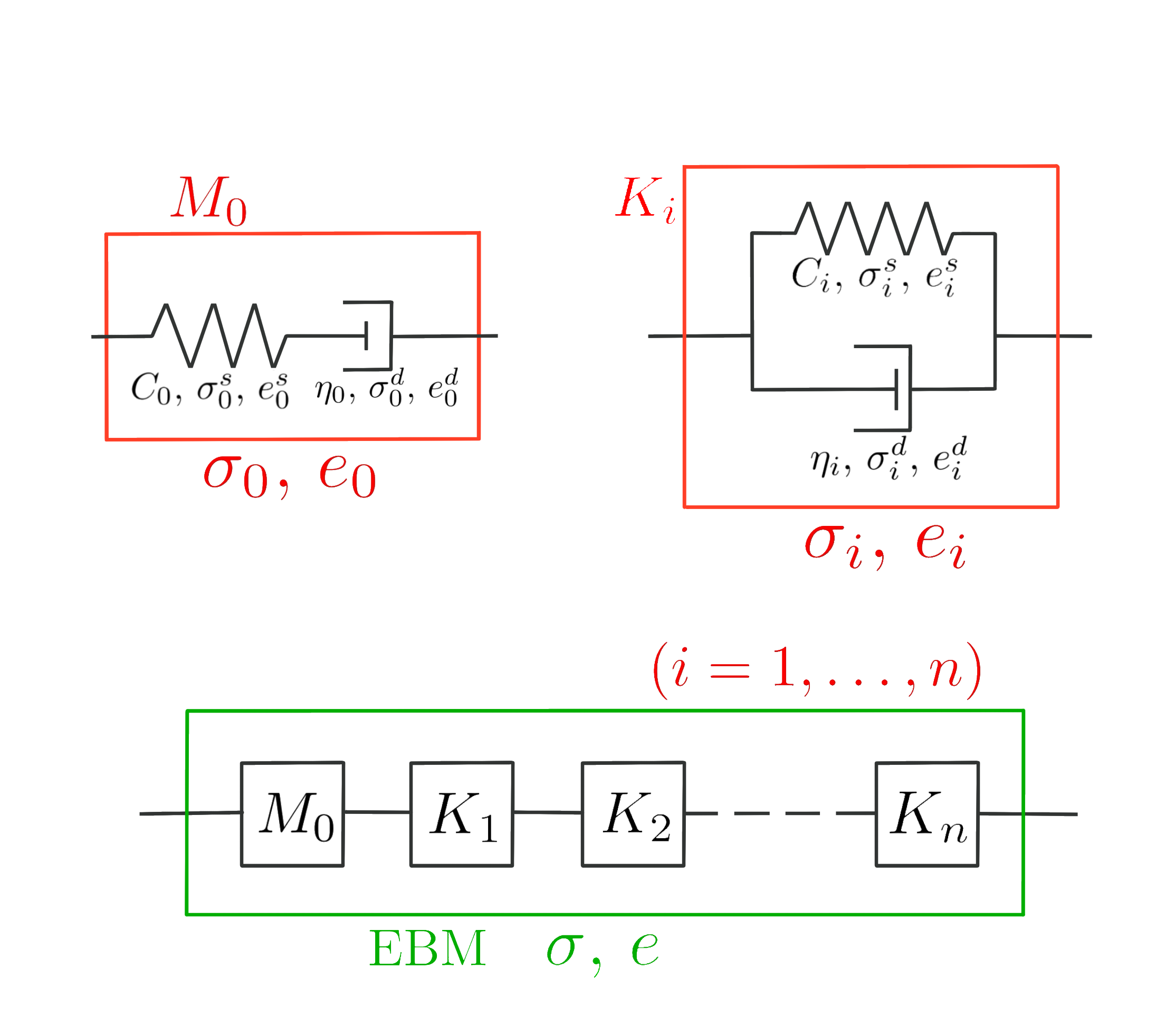}}
\caption*{The extended Burgers model consists of one Maxwell model $M_0$ and $n$ numeber of the Kelvin-Voigt models $K_i$, 
$i=1,\cdots,n.$ Also, they are connected in series. Further, the superscripts $s$ and $d$ refer to springs and dashpots, respectively.}
    \label{fig:EBM}
\end{figure*}

\end{document}